\documentclass{amsart}
\usepackage{amsmath,amssymb,amsfonts,amsthm,amscd, mathabx}
\usepackage{amssymb}
\usepackage{pgf}
\usepackage{todonotes}
\usepackage{epsfig}
\usepackage{mathrsfs}
\usepackage{verbatim}
\newtheorem{theorem}{Theorem}[section]
\newtheorem{lemma}[theorem]{Lemma}
\newtheorem{proposition}[theorem]{Proposition}

\theoremstyle{definition}

\theoremstyle{remark}

\numberwithin{equation}{section}

\renewcommand{\j}{\jmath}

\newcommand{\R}{\mathbb{R}}
\newcommand{\N}{\mathbb{N}}

\renewcommand{\ker}{\mathrm{Ker}}
\renewcommand{\mod}{/}

\begin{document}

\title[Generalized measure algebras]{On natural density, orthomodular lattices, measure algebras and non-distributive $L^p$ spaces}

\author{Jarno Talponen}
\address{University of Eastern Finland, Department of Physics and Mathematics, Box 111, FI-80101 Joensuu, Finland}
\email{talponen@iki.fi}

\keywords{Orthomodular lattice, Boolean algebra, measure algebra, group-valued submeasure, natural density, density set, content extension, Banach space, function space, $L^p$ space}
\subjclass[2010]{06C15, 28A60, 11B05, 46E99}
\date{\today}

\begin{abstract}
In this note we first show, roughly speaking, that if $\mathcal{B}$ is a Boolean algebra included in the natural way in the collection 
$\mathcal{D}/_\sim$ of all equivalence classes of natural density sets of the natural numbers, modulo null density, then $\mathcal{B}$ extends to a $\sigma$-algebra $\Sigma \subset \mathcal{D}/_\sim$ and the natural density is $\sigma$-additive on $\Sigma$. We prove the main tool employed in the argument in a more general setting, involving a kind of quantum state function, more precisely,  a group-valued submeasure on an orthomodular lattice. At the end we discuss the construction of `non-distributive $L^p$ spaces' by means of submeasures on 
lattices.
\end{abstract}

\maketitle
\newcommand{\M}{\mathbb{M}}
\section{Introduction}
This article deals with abstract versions of the measure algebra and natural density notions.
The considerations here are much in the spirit of Galois connections.

Let us denote by $\mathcal{D}$ the collection of all \emph{density sets}, i.e. sets $A \subset \N$ such that the \emph{natural density}
\[d(A):=\lim_{n\to\infty} \frac{|A\cap \{1,\ldots ,n\}|}{n}\]
exists. We denote by $\mathcal{N} \subset \mathcal{D}$ the collection of all null density sets, i.e. sets $A$ with $d(A)=0$.
We denote by $\sim$ the equivalence relation on $\mathcal{D}$ given by $K\sim M$ if the symmetric difference of $K \vartriangle M$ is in $\mathcal{N}$. Consider the set $\mathcal{D} /_\sim$ of equivalence classes $[x]$ 
modulo null density. It has the natural partial order given by $[x] \preceq_{\mathcal{N}} [y]$ if there is $N \in \mathcal{N}$ such that $x \subset y \cup N$. 

The natural density is of course an important notion 
in number theory. For instance, recall Szemer\'edi's theorem in \cite{Sz} which states that every positive 
(i.e. non-null) density set contains a $k$-arithmetic progression for every $k \in\N$.

Regarding the natural density as a content in the sense of measure theory is a classical theme, 
see e.g. \cite{Buck,Buck2}. The following theorem is the main result in this paper:
\begin{theorem}\label{thm: 1}
Let $\mathcal{F}\subset \mathcal{D}$ be a family closed under finite intersections. 
Then there is a $\sigma$-algebra $\Sigma$ order-isomorphically included in $\mathcal{D} /_\sim$ 
such that $\mathcal{F} /_\sim \subset \Sigma$ and $\hat{d} \colon \Sigma \to [0,1]$, 
$\hat{d}(K /_\sim )=d(K)$, is $\sigma$-additive. 

Morerover, if $\mathcal{F}/_\sim$ is countable and the corresponding $\sigma$-generated measure algebra $(\Sigma,   \hat{d})$ is atomless, then it is in fact isomorphic to the measure algebra on the unit interval.
\end{theorem}

In the first part of the statement $\Sigma$ can be viewed essentially as a family of density sets closed under finite intersections and therefore the conditions in the first part of the theorem are in fact equivalent.
What is highly counter-intuitive about the latter part of the result is that the algebra is modeled on $\N$ 
which clearly does not admit an atomless $\sigma$-additive probability measure. 

In the classical papers on this topic \cite{Buck,Buck2} equivalence classes were considered essentially modulo finite subsets of the natural numbers, instead of taking equivalence classes modulo null-density. Presumably,
the embeddings of Boolean algebras in $\mathcal{D} / \mathrm{Fin}$ behave differently compared 
to such embeddings in $\mathcal{D} /_\sim$. 
Also, the question of embedding the measure algebra $\M$ into $\mathcal{P}(\N)/\mathrm{Fin}$ is very 
delicate, see e.g. \cite{blass,dow,DiNasso}.

In fact we will prove the main tool applied in the above a result in a more general form. In this general version
we will replace Boolean algebras with orthomodular lattices and the measures are replaced by kind of 
quantum probability measures or group-valued submeasures, cf. \cite{l}. Ortomodular lattices are closely related to the formulation of quantum logic in physics. In this connection the above mentioned measure is 
sometimes called the state function. This mapping also appears to have a connection to graded semi-modular lattices. From the functional analysis point of view the orthomodular lattices model a system of closed subspaces of a Hilbert space with natural interpretations involving the complemented lattice operations.  

The end of the paper is devoted to studying $L^p$ spaces over `non-distributive measure algebras', or rather $L^p (\mathcal{L},\varphi)$ where 
$\mathcal{L}$ is a bounded lattice and $\varphi$ is an order-preserving subadditive mapping $\varphi \colon \mathcal{L} \to [0,1]$ with $\varphi({\bf 0})=0$ 
and $\varphi({\bf 1})=1$. The non-commutative $L^p$ spaces are under active investigation (cf. \cite{nonComm}) and the above generalization of the classical $L^p$ spaces goes to another direction. We will adopt 
an approach which somewhat resembles the construction of some Banach tensor products, 
cf. \cite{fremlin_mathann}, \cite{Wittstock}.

\section{Group-valued submeasures on lattices and extensions of the natural density}

We refer to the monographs in the references for suitable background information. 
As mentioned above, our aim is to generalize the main tool applied in the proof of Theorem \ref{thm: 1} simultaneously to several directions. We list the assumptions and conventions imposed in this section:
\begin{itemize}
\item In what follows $\mathcal{L}$ is a bounded countably complete orthomodular lattice. 
The minimal and maximal elements are denoted by $\bf 0$ and $\bf 1$, respectively. 
\item Let $G$ be a partially ordered locally compact topological group. Its neutral element is denoted by $e$.
We impose the following conditions:
\begin{enumerate}
\item[(1)]{If $x\leq y$ then $zx\leq zy$ and $xz\leq yz$,}
\item[(2)]{If $x\leq y$ then $y^{-1}\leq x^{-1}$.}
\item[(3)]{The topology is stronger than the order topology.}
\item[(4)]{The topology of $G$ has countable character.} 
\end{enumerate}
\item Let $I$ be an infinite index set. 
\item Let $\mathcal{F}$ be a filter on $I$ which is countably incomplete in the sense that there is a 
sequence $(F_n ) \subset \mathcal{F}$ with $\bigcap_n F_n = \emptyset$. 
\item For each $i\in I$ we let 
$m_i \colon \mathcal{L} \to G$ be a mapping satisfying the following properties:
\begin{enumerate}
\item[(i)]{$m_i(0)=e$.}
\item[(ii)]{$m_i$ is order-preserving.}
\item[(iii)]{$m_i (x\vee y)\leq m_i (x) m_i (y)$ for $x,y\in \mathcal{L}$.}
\item[(iv)]{If $x \leq y$ then $m_i (y)=m_i (x) m_i (y\wedge x^\bot )$.}
\end{enumerate}
\item We assume $\Lambda \subset \mathcal{L}$ is the maximal subset such that the mapping 
$m\colon \Lambda\to G$ given by
\[m(x)=\lim_{i,\mathcal{F}}m_i (x),\quad x\in \Lambda,\]
is defined.
\item This yields a natural ideal $\mathcal{N}=\{x\in \mathcal{L}:\ m(x)=e\} \subset \mathcal{L}$.
\item We define a preorder on $\mathcal{L}$ by $x \sqsubseteq_\mathcal{N} y$ if there is $N \in \mathcal{N}$ such that $x \leq y \vee N$.
\item We define an equivalence relation on $\mathcal{L}$ by $x \sim y$ if $x \sqsubseteq_\mathcal{N} y$ and $y \sqsubseteq_\mathcal{N} x$.
\item Then $\sim$ is  a $\vee$-semilattice congruence and we have a natural quotient mapping $\mathcal{L} \to \mathcal{L}/ \mathcal{N}$, $x \mapsto [x]$, which is a 
$\vee$-semilattice epimorphism.
\item We define $\hat{m}\colon \mathcal{L}\mod \mathcal{N}\to G$ by $\hat{m}([x])=m(x)$.
(This is well-defined by the assumptions on $m$.)
\item We assume that for each $m_i$ there is a \emph{support} $s_i \in \mathcal{L}$ in the sense that 
\[m_i (x)=m_i(x\wedge s_i )\quad \forall x\in \mathcal{L}.\]
\item We assume that if $\Gamma \subset I$ is such that $I \setminus \Gamma \in \mathcal{F}$ then $\bigvee_{i\in\Gamma} s_i \in \mathcal{L}$ exists and is included in $\mathcal{N}$. 
\item We assume that if $x,y\in \Lambda$, $x \sqsubseteq_\mathcal{N} y$, then $x \wedge y \in \Lambda$ and $[x\wedge y ]=[x]$.
\item If $x,y\in \Lambda$, $x \sqsubseteq_\mathcal{N} y$, $m(x)=m(y)$, then $[x]=[y]$.
\end{itemize}

Admittedly, this list of assumptions long and appears technical. It is not very restrictive, for example 
the density function $d$ fits easily to this framework. The point here is stretching the generality in which 
our main tool holds to the limit, as there appears to be applications for such structures,
see \cite{beran, l, DiNasso}. 

The proof of the following auxiliary fact is subsequently applied in the argument of the main result.

\begin{lemma}\label{lm: main}
With the above notations the mapping $\hat{m}$ is countably additive in the following sense: 
Suppose that $([x_n ]) \subset  \Lambda\mod \mathcal{N}$ is an increasing sequence such that 
$\bigvee_n \hat{m}([x_n ])$ exists. Then $\bigvee_n  [x_n ] \in  \Lambda\mod \mathcal{N}$ exists
and $\hat{m}(\bigvee_n  [x_n]  )=\bigvee_n \hat{m}([x_n ])$. 
\end{lemma}

\begin{proof}
Let $G\supset U_1 \supset U_2 \supset \ldots \supset U_n \supset \ldots$, $n\in\N$, be a countable neighborhood basis of $e$.  
Let $(F_n ) \subset \mathcal{F}$ be a decreasing sequence such that $\bigcap_{n\in \N} F_n = \emptyset$. Let 
$\Gamma_1 \subset F_1$ be the subset of all indices $i$ such that 
\[m_i (x_1 ) \in m(x_1 ) U_1 .\]
Next, we let $\Gamma_2 \subset \Gamma_1 \cap F_2$ be the subset of all indices $i$ such that 
\[m_i (x_2 ) \in m(x_2 ) U_2 .\]
We proceed recursively in this fashion: for each $n+1\in\N$ we let 
$\Gamma_{n+1} \subset \Gamma_n \cap F_n$ be the subset of all indices $i$ such that 
\[m_i (x_{n+1} ) \in m (x_{n+1})U_{n+1} .\]
Note that by the construction of $m$ we have that $\Gamma_n \in \mathcal{F}$ for each 
$n\in\N$. 

Let $s_i $ be the supports of the mappings $m_i$. We put 
\[z_n := \bigvee_{i \in I \setminus \Gamma_n} s_i ,\quad n\in\N .\]
Note that the orthomodular law yields that 
$x\vee z_n = z_n \vee (z_{n}^\bot \wedge (x \vee z_n ))$ for each $x \in \mathcal{L}$. 
Since $m(z_n ) =e$ by the assumptions, we may modify recursively the selection of the sets $\Gamma_{n}$
to obtain a decreasing sequence $(\Gamma_{n}' )$ with $\Gamma_n \supset  \Gamma_{n}' \in \mathcal{F}$
in such a way that 
\[m_i (z_{n}^\bot \wedge (x_{n+1}\vee z_n )) \in m_i ( x_{n+1})U_{n+1} \]
and 
\[m_i (x_{n+1} ) \in m (x_{n+1})U_{n+1}\]
for all $i \in \Gamma_{n+1}'$, $n \in\N$.

Put 
\[w_n = \bigvee_{1\leq j \leq n} z_{j}^\bot \wedge (x_{j+1} \vee z_j ),\quad n\in\N.\]
By using the countable completeness of $\mathcal{L}$ we may put
\[w = \bigvee_n w_n .\]

Note that if we choose a sequence $(i_n )$ with $i_n \in \Gamma_n$ then 
$m(x_{n})^{-1} m_{i_n} (w) \to e$ as $n\to \infty$ and that  
$M=\bigvee_n m(x_{n})$ exists by assumption. Thus by the assumptions involving the topology of $G$ 
we obtain that $(m_{i_n} (x_{n})) \subset G$ contains a subsequence converging to $M$.
From the arbitrary nature of the selection of $(i_n )$ we conclude that $\lim_{i,\mathcal{F}} m_i (w) = M$.
In particular, $w \in \Lambda$. 

Note that 
\[[x_{n+1} ] = [z_n \vee (z_{n}^\bot \wedge (z_n \vee x_{n+1}) )] ,\]
so that
\[[w]= [z_n \vee w] \geq [z_n \vee w_n ] \geq [x_n ] .\]
This means that $[w]$ is an upper bound for the sequence $([x_n ] )$.

Assume next that $[y] \in  \Lambda\mod \mathcal{N}$ is some upper bound  for the sequence $([x_n ] )$. Then according to the assumptions $y \wedge w_n $ exists because $[w_n ]=[x_n ]$ and it defines an increasing sequence. By the countable completeness of $\mathcal{L}$ we may define
$w_0 :=\bigvee_n   y \wedge w_n$. Clearly $w_0 \sqsubseteq_\mathcal{N} y$ and $w_0 \sqsubseteq_\mathcal{N} w$. Moreover, $m(w_0 )= M$, since $w_0 \leq w$ and 
$m(y\wedge w_n ) \to M$ as $n\to\infty$. It follows from the assumptions that 
$[w_0 ] = [w]$. Thus $[w]$ is the least upper bound for the sequence $([x_n ])$.  

\end{proof}

\section{The main result}
In this section we will give the proof for Theorem \ref{thm: 1}, the main result.
\subsection{Pointless $L^p$ spaces}

We will introduce subsequently a way of defining $L^p$ spaces over a general lattice. This construction can be specialized to obtain $L^p$ spaces 
over measure algebra, i.e. $L^p (\Sigma,\mu)$ (cf. \cite[Ch. 36]{fremlin}). These spaces can be built by defining a suitable norm 
for `simple functions' $\sum_k a_k \otimes M_k$ where $a_k \in \R$ and $M_k \in \Sigma$ are essentially pairwise disjoint. Then the norm is given by 
\[\left \|\sum_k a_k \otimes M_k \right\| = \left(\sum_k |a_k |^p \mu(M_k )\right)^{\frac{1}{p}}.\]
This space is then completed to get the required space $L^p (\Sigma,\mu)$.
(The case with a lattice, instead of a Boolean algebra is more complicated.) One can see that this space has a natural Banach lattice structure 
and it admits a Lebesgue integral like functional.

\subsection{The proof}
\begin{proof}[Proof of Theorem \ref{thm: 1}] 
Let us resume the notations of the theorem.
It is well-known that if $A$ is a density set and $N$ a null-density set, then $A\cup N$ and $A \setminus N$ 
are also density sets and have the same density as $A$. Therefore $\hat{d}$ in the theorem is well-defined.

We will apply the argument of the Dynkin-Sierpinski  $\pi$-$\lambda$-lemma.
We call a subset $\Delta \subset \mathcal{D}/_\sim$ a $d$-system if it satisfies the following conditions:
\begin{enumerate}
\item[(i)] $[\N] \in \Delta$;
\item[(ii)] For each $[A] , [B]\in \Delta$, $[A] \preceq_{\mathcal{N}} [B]$, we have $[B \setminus A] \in \Delta$; 
\item[(iii)] For each $[A] , [B]\in \Delta$, $A\cap B = \emptyset$, we have $[A\cup B] \in \Delta$;
\item[(iv)] If $(A_n ) \subset \Delta$ is an increasing sequence in the order inherited from 
$\mathcal{D}/_\sim$, then the least upper bound $A$ for this sequence exists in $\mathcal{D}/_\sim$
and moreover $A \in \Delta$.
\end{enumerate} 

We claim that $\mathcal{D}/_\sim$ is a $d$-system itself. Indeed, the conditions (i)-(iii) follow 
immediately from the well-known properties of the natural density. The last condition follows from the proof of 
Lemma \ref{lm: main}.

Let $\Delta$ be the intersection of all $d$-systems in $\mathcal{D}/_\sim$ containing 
$\mathcal{F}/_\sim$. 
Put 
\[\mathcal{A}_1 = \{[A] \in \Delta \colon [A \cap F] \in \Delta\ \forall F \in \mathcal{F}\}.\]
We claim that $\mathcal{A}_1 $ is a $d$-system. 
Indeed, since $\mathcal{F}$ is closed under finite intersections, we obtain that 
$\mathcal{F} /_\sim \subset \mathcal{A}_1$. 
Therefore, $\mathcal{A}_1$ satisfies (i). It is easy to see that $\mathcal{A}_1$ satisfies (iii)-(iv). 
Note that if $[A \cap F] \in \Delta$ for a given $A \in \Delta$ and $F \in \mathcal{F}$, then 
$[(\N \setminus A ) \cap F] = [F \setminus (A \cap F)]\in \Delta$ according to (ii).
Thus, by using (iii) we obtain that $\mathcal{A}_1$ satisfies (ii).
Consequently, $\mathcal{A}_1$ is a $d$-system and it follows that $\mathcal{A}_1 = \Delta$.

Put 
\[\mathcal{A}_2 = \{[A] \in \Delta \colon [A \cap C] \in \Delta\ \forall C \in \Delta\}.\]
By studying $\mathcal{A}_1 =\Delta$ we obtain that $\mathcal{F}/_\sim \subset \mathcal{A}_2 $.
We easily see again that $\mathcal{A}_2 $ is a $d$-system. 
It follows that $\mathcal{A}_2 =\Delta$. Thus $\Delta$ satisfies that  
$[A\cap B] \in \Delta$ for any $[A],[B] \in \Delta$.

Now we may define a Boolean algebra structure on $\Delta$ as follows:
$[A] \vee [B] = [A \cup B] = [(A\setminus B) \cup (A \cap B ) \cup (B \setminus A)]$, 
$[A] \wedge [B] = [A \cap B]$ and $\neg [A]=[\N \setminus A]$. These operations 
satisfy the axioms of a Boolean algebra since the corresponding operations on the subsets of $\N$ satisfy 
them. It follows from the last condition of the $d$-system that $\Delta$ forms in fact a $\sigma$-algebra.

Now, it is a basic well-known fact that $\hat{d}$ is finitely additive on $\Delta$ in the sense that if 
$A, B \in \mathcal{D}$ are disjoint, then 
$\hat{d} ([A] \vee [B])=\hat{d} ([A \cup B])=\hat{d} ([A]) + \hat{d} ([B])$ . 
It follows from the proof of Lemma \ref{lm: main} that $\mu=\hat{d}$ is $\sigma$-additive on $\Sigma=\Delta$. 

It remains to show the latter part of the statement. Let us study the space $L^1 (\Sigma, \mu)$. 
Since any measure algebra can be represented by using the Stone space as a measure algebra of a measure space (see \cite[321J]{fremlin}), the above space is actually isometrically order-isomorphically a genuine function space $L^1 (\Omega, \Sigma, \mu)$ where $\mu$ is a probability measure. 
If $\mathcal{F} /_\sim$ is countable and 
$\Sigma$ is $\sigma$-generated by $\mathcal{F} /_\sim$, then it is clear that $L^1 (\Sigma, \mu)$ 
is separable as the simple functions generated by $\mathcal{F} /_\sim$ are dense (e.g. by the Martingale Convergence Theorem). Note that we assumed the measure algebra is atomless. Therefore  Maharam's classification of measure algebras implies that $(\Sigma, \mu)$ is isomorphic to $\mathbf{M}$, the Maharam's space corresponding to the unit interval with the completed Lebesgue measure. 
Indeed, if $(\Sigma, \mu)$ contained a part isomorphic to the measure space on $\{0,1\}^{\omega_1 }$, generated by $\aleph_1$-many i.i.d. fair Bernoulli trials (i.e. coin tosses), then $L^1 (\Sigma, \mu)$ would be non-separable.
\end{proof}

\section{Non-distributive $L^p$ spaces}
Let us consider a bounded lattice $\mathcal{L}=(\mathcal{L}, \vee,\wedge, {\bf 0}, {\bf 1} )$ (which need not be modular or countably complete). 
Let $\varphi \colon \mathcal{L} \to [0,1]$ be an order-preserving mapping such that $\varphi({\bf 0})=0$, $\varphi({\bf 1})=1$ and 
\[\varphi(A \vee B)\leq \varphi(A) + \varphi(B)\quad \forall A,B \in\mathcal{L} .\]

Consider the space $c_{00}(\mathcal{L})$ and denote its canonical Hamel basis unit vectors by 
$e_{A}$, $A\in \mathcal{L}$.
Let $\Delta \subset c_{00}(\mathcal{L})$ be the linear subspace given by 
\[\Delta = [e_{{\bf 0}}] + \mathrm{span} ((e_{A} + e_{B}) - (e_{A\vee B} + e_{A\wedge B})\colon A,B \in \mathcal{L}\}.\]

We let 
\[q\colon c_{00}(\mathcal{L}) \to c_{00}(\mathcal{L}) / \Delta\]
be the canonical quotient mapping. Write $X= c_{00}(\mathcal{L}) / \Delta$ and 
we denote by 
\[a\otimes A := q(ae_A) \in X,\quad a\in\R,\  A \in \mathcal{L}.\] 
Note that $X$ is the space of vectors of the form 
\[\sum_{i\in I} a_i \otimes A_i,\quad a_i \in \R,\ A_i \in  \mathcal{L},\ I\ \mathrm{finite}.\]
 
\begin{proposition}
Let $\mathcal{L}$ be an orthomodular lattice and $X$ be the corresponding space defined as above. Then 
for any $a \otimes A ,b\otimes B \in X$ it holds that
\[a \otimes A + b \otimes B = a \otimes (A \wedge (A\wedge B)^\bot ) + (a+b)\otimes (A\wedge B) 
+ b\otimes (B \wedge (A\wedge B)^\bot ).\] 
Any element $x \in X$ can be represented as 
\[x=\sum_i a_i \otimes A_i \]
where $A_i \wedge A_j ={\bf 0}$ for $i\neq j$.
Moreover, if $(\Omega,\Sigma,\mu)$ is a measure space and $\mathcal{L}=\Sigma$, 
then the space $\mathcal{S}$ of simple functions on $(\Omega,\Sigma,\mu)$ can be linearly identified with 
$X$ as follows:
\[\sum_i a_i 1_{A_i }\longmapsto  \sum_i a_i \otimes A_i .\]
\end{proposition}  
\begin{proof}
Note that the orthomodularity condition yields
\[A = (A \wedge (A\wedge B)^\bot ) \vee (A \wedge B),\quad B = (B \wedge (A\wedge B)^\bot )\vee (A \wedge B).\]
It follows from the construction of $X$ that 
\[a \otimes A = a \otimes (A \wedge (A\wedge B)^\bot ) + a \otimes (A \wedge B),\]
\[b \otimes B = b \otimes (B \wedge (A\wedge B)^\bot ) + b \otimes (A\wedge B).\]
It follows from the construction of $X$ that 
\[a \otimes (A \wedge B) +  b \otimes (A\wedge B) =  (a+b) \otimes (A\wedge B)\]
and the first part of the claim follows.

The second part of the statement thus follows by constructing a refined system of 
elements $A_i$ where one uses recursively the above type decompositions. 
The last part of the statement is then easy to see by studying a maximal linearly independent system
$\{1_{A_\beta}\}_\beta \subset \mathcal{S}$.
\end{proof}

Let $\sqsubseteq$ be the preorder on $X$ generated by the following conditions:
\begin{enumerate}
\item[(i)] $\sqsubseteq$ satisfies the axioms of a partial order of a vector lattice, except possibly anti-symmetry; 
\item[(ii)] $a \otimes A \sqsubseteq b \otimes A$ whenever $a\leq b$; 
\item[(iii)] $1 \otimes A \sqsubseteq  1 \otimes B$ whenever $A\leq B$ in the intrinsic partial order of $\mathcal{L}$.
\end{enumerate} 
 
We define a semi-norm on $X$ by 
\[\left\|\sum_i a_i \otimes A_i \right\|= \inf \left\{\left(\sum_k |b_k |^p \varphi(B_k )\right)^{\frac{1}{p}} \colon \pm\sum_i a_i \otimes A_i \sqsubseteq  \sum_k |b_k |  \otimes B_k \right\}.\]
Indeed, it is easy to see that this is a semi-norm; the triangle inequality follows from the 
condition that $x+y\sqsubseteq v +w$ whenever $x\sqsubseteq v$ and $y\sqsubseteq w$.

We let $\mathcal{X}=X/ \ker \|\cdot\|$ and we still denote by $\|\cdot\|$ the obvious norm induced 
by the above semi-norm on this quotient space. We also take the quotient of the elements $a\otimes A$
and continue using the same notation. Dividing by the kernel corresponds to the operation of identification of functions 
that coincide a.e. in the classical setting. However, we do not know if there is an additional construction-specific necessity for performing this operation.

We let $L^p (\mathcal{L} ,\varphi)$ be the Banach space obtained as the completion of the normed space 
$(\mathcal{X}, \|\cdot\|)$. We define a relation $\leq$ on $L^p (\mathcal{L} ,\varphi)$ by 
means of the closure of the positive  cone given by the $\sqsubseteq$ relation.

\begin{proposition}
The space $L^p (\mathcal{L} ,\varphi)$ endowed with the order $\leq$ is an ordered vector space with the 
property that if $0\leq x \leq y$ then $\|x\|\leq \|y\|$. 
\end{proposition}
\begin{proof}
First note that the relation $\sqsubseteq$ is given by a convex cone, since we imposed that 
$\sqsubseteq$ satisfies the vector lattice conditions. Typically there exist elements $x \in X$, $x\neq 0$, 
such that $\pm x \sqsubseteq 0$.

If $x \leq y$, then by the definition of the $\leq$ order, $y-x \in C_+$ where the convex cone $C_+$ is the closure 
\[C_+ :=\overline{\{x\in \mathcal{X}\colon 0 \sqsubseteq x\}} \subset L^p (\mathcal{L} ,\varphi) .\]
It is clear that the relation $\leq$ is a vector lattice order if $\leq$ is an anti-symmetric relation. 
This is the case exactly when the positive cone $C_+$ is salient. 

Note that in any case for $0\leq x \leq y$ it holds that $0 \leq y-x$ and therefore there is 
a sequence $(z_n )$ in the positive cone of $\sqsubseteq$ such that $\| z_n  -(y-x)\|_{ L^p (\mathcal{L} ,\varphi)} \to 0$ as $n\to\infty$.  By the definition of the norms we have that $\|y\| \geq \|x\|$. 

To verify the saliency of $C_+$, assume that $\pm z \in C_+$. Then $0 \leq \pm z \leq 0$. Thus 
$\|z\|_{L^p (\mathcal{L} ,\varphi)}=0$ by the previous observation. Since we defined $\|\cdot \|_{L^p (\mathcal{L} ,\varphi)}$ as a completion of the norm of $\mathcal{X}$ (instead of working with $X$), we obtain that $z=0$ which shows the saliency of $C_+$.
\end{proof}

\subsection{Examples} Here finite sequence spaces are considered with coordinate-wise order.
Below we will write \emph{all} order relations explicitly visible.
If $\mathcal{L}=\{{\bf 0},A,B,{\bf 1}\}$ with ${\bf 0} \leq A,B \leq {\bf 1}$ (thus $A\wedge B = {\bf 0}$, 
$A\vee B = {\bf 1}$) and $\varphi(A)=\varphi(B)=\frac{1}{2}$, then $L^p (\mathcal{L},\varphi)$ is order-isomorphically isometric to $\ell^p (2)$ for $1\leq p <\infty$ (the $2$-dimensional $\ell^p$ space). 

If $\mathcal{L}=\{{\bf 0},A,B,C,{\bf 1}\}$ with ${\bf 0} \leq A,B,C \leq {\bf 1}$ (i.e. smallest non-distributive but modular lattice) and $\varphi(A)=\varphi(B)=\varphi(C)=\frac{1}{2}$, then $L^p (\mathcal{L},\varphi)$ is $1$-dimensional. Indeed, write pairwise distinct $i,j,k \in \{1,2,3\}$ and $D_1 = A$, $D_2 = B$, $D_3 =C$, then 
we have
\begin{multline*}
1\otimes D_i - 1\otimes D_j = (1\otimes D_k + 1\otimes D_i) - (1\otimes D_k + 1\otimes D_j)\\
=(1\otimes {\bf 1} - 1\otimes  {\bf 0}) - (1\otimes {\bf 1} - 1\otimes  {\bf 0})\\ 
= 1\otimes {\bf 1} - 1\otimes {\bf 1} = 0 \otimes {\bf 1} =0.
\end{multline*}
Thus $1\otimes D_i = 1\otimes D_j = 1 \otimes A$. Hence
\[1\otimes {\bf 1} = 1\otimes D_i + 1\otimes D_j = 1\otimes A + 1\otimes A = 2 \otimes A.\]
This means that the space is spanned by $1\otimes A$.

If $\mathcal{L}=\{{\bf 0},A,B,C,{\bf 1}\}$ with ${\bf 0} \leq A \leq {\bf 1}$, ${\bf 0}\leq B \leq C \leq {\bf 1}$ (i.e. smallest non-modular lattice) and $0< \varphi(A)$, $0<\varphi(B)\leq \varphi(C)$, then $L^p (\mathcal{L},\varphi)$ is isometric to $\ell^p (2)$. The linear 
isometry is given by 
\[1\otimes A \mapsto \phi(A)^{\frac{1}{p}}\ e_1 ,\quad 1\otimes B \mapsto \phi(B)^{\frac{1}{p}}\ e_2 .\]
Indeed, 
\[1 \otimes B - 1 \otimes C = (1 \otimes A + 1\otimes B)- (1\otimes A + 1\otimes C)=1\otimes {\bf 1}
-1\otimes {\bf 1} =0,\]
so that the linear space reduces to the first example. In the definition of the norm we clearly choose 
$B$ instead of $C$, since $\varphi(B)\leq \varphi(C)$.

If $(\Sigma/_\sim , m)$ is the measure algebra on the unit interval, then 
$L^p (\mathcal{L},\varphi)$ with $\mathcal{L}=\Sigma/_\sim$, $\varphi=m$, is order-isomorphically isometric to $L^p$. The isomorphism is given by the 
extension of 
\[\sum_i a_i \otimes [A_{i}]_\sim  \quad \longmapsto \quad \left[\sum_i a_i 1_{A_{i}} \right]_{\stackrel{\mathrm{a.e.}}{=}} .\] 

\subsection{Complementation and function spaces}
Suppose that $\mathcal{L}$ is additionally an orthomodular lattice. 
Then for each $M\in \mathcal{L}$ we may define norm-$1$ linear projections $P_M , Q_M \colon L^p (\mathcal{L} ,\varphi) \to L^p (\mathcal{L} ,\varphi)$ as follows. 
First we choose a Hamel basis $\{1\otimes N_\alpha \}_\alpha$ with $N_\alpha \leq M$ or 
$N_\alpha \leq M^\bot$ of $X$  (i.e. a maximal linearly independent family) and then we study operators $X \to X$ as follows : 
\[P_M \colon \sum_i a_i \otimes N_{\alpha_i} \mapsto \sum_i a_i \otimes (M\wedge N_{\alpha_i}) ,\] 
\[Q_M \colon  \sum_i a_i \otimes N_{\alpha_i} \mapsto \sum_i a_i \otimes (M^\bot \wedge N_{\alpha_i}) .\] 
By the orthomodularity of $\mathcal{L}$ we get $(M\wedge N_{\alpha_i}) \wedge (M^\bot \wedge N_{\alpha_i}) =0$ and 
$(M\wedge N_{\alpha_i}) \vee (M^\bot \wedge N_{\alpha_i}) = N_{\alpha_i}$ above. Thus, by the construction 
of $X$ we get 
\[P_M Q_M = Q_M P_M \colon X \to \{0\},\quad P_M + Q_M = Id.\]
Note that the preorder can be alternatively generated by replacing the sets $A$ and $B$ with sets of the
type $A \wedge M$, $B \wedge M$, $A \wedge M^\bot$ and $B \wedge M^\bot$. Since $P_M$ 
and $Q_M$ are linear and preserve the conditions $1 \otimes A \sqsubseteq 1 \otimes B$ and 
$a \otimes A \sqsubseteq b \otimes A$, we can see that $x \sqsubseteq y$, $x,y \in X$, 
if and only if $P_M (x) \sqsubseteq P_M (y)$ and $Q_M (x) \sqsubseteq Q_M (y)$.

Thus, applying $P_M$ and $Q_M$ to the condition  
$\pm\sum_i a_i \otimes A_i \sqsubseteq  \sum_k |b_k |  \otimes B_k $, appearing in the definition of the 
semi-norm, we get that both $P_M$ and $Q_M$ are bicontractive projections. Thus they extend as 
bicontractive projections on $L^p (\mathcal{L} ,\varphi)$.
Moreover, if $\varphi(A \wedge M)+\varphi(A \wedge M^\bot )=\varphi(A)$ should hold for all $A \in \mathcal{L}$, then
\begin{equation}\label{eq:  psum}
\|x\|_{L^p (\mathcal{L} ,\varphi)}^p =\|P_M x\|_{L^p (\mathcal{L} ,\varphi)}^p +\|Q_M x\|_{L^p (\mathcal{L} ,\varphi)}^p ,\quad 1\leq p<\infty ,\ x\in L^p (\mathcal{L} ,\varphi) .
\end{equation}

\begin{theorem}
Let $1\leq p <\infty$, $\mathcal{L}$ be an orthomodular lattice with an order-preserving map 
$\varphi \colon \mathcal{L} \to [0,1]$, as above. Let $(\Omega, \Sigma, \mu)$ be a probability space and 
$\Sigma_0 \subset \Sigma$ a Boolean algebra which $\sigma$-generates $\Sigma$. 
Let us assume that $\varphi (M \vee N) = \varphi (M)+ \varphi (N)$ whenever $N \leq M^\bot$.
Suppose that there is an order-embedding $\j \colon \Sigma_0 \to \mathcal{L}$ such that $\mu(M)=\varphi(\j M)$ for all $M \in \Sigma_0$. (We are not assuming here that $\j$ respects the orthocomplementation operation.) Then 
\[\sum_i a_i 1_{A_i} \mapsto \sum_i a_i \otimes \j (A_i ) ,\quad A_i \in \Sigma_0\]
extends to a linear isometry $L^p (\Omega,\Sigma, \mu) \to L^p (\mathcal{L} ,\varphi)$.
\end{theorem}
\begin{proof}
First, by considering the order-embedding and the definition of $X$ compatible with operations of simple functions on the Boolean algebra we can show recursively that the values on the right side remain invariant in passing to a form with $A_i$:s pairwise $\mu$-almost disjoint.  
Thus it is easy to see from the construction of $X$ that the above linear map is in fact well defined.

Note that the special simple functions of the above form, $\sum_i a_i 1_{A_i} \in 
L^p (\Omega,\Sigma, \mu) $, are dense in the set of all simple functions, since $\Sigma_0$ $\sigma$-generates 
$\Sigma$. Consequently, these special simple functions are dense in $L^p (\Omega,\Sigma, \mu) $. 
On the other hand, since $L^p (\mathcal{L} ,\varphi)$ is complete by the construction, we are only required 
to verify that the linear mapping appearing in the statement is norm-preserving.

Thus, pick an element $\sum_i a_i \otimes \j (A_i )$ in $X$. Let $\varepsilon>0$.
Let $\sum_j |b_j|  \otimes B_j \in X$ such that
\begin{equation}\label{eq: pm}
\pm \sum_i a_i \otimes \j (A_i ) \sqsubseteq \sum_j |b_j|  \otimes B_j 
\end{equation}
and
\[\left\| \sum_i a_i \otimes \j (A_i )\right\|^p \geq \sum_j |b_j|^p  \varphi( B_j )  -\varepsilon .\]

By the beginning remarks of the proof we may assume without loss of generality that $A_i \wedge A_j =0$ 
above for $i\neq j$. Actually, we may refine the above representations of elements of $X$.
Namely, by using the orthomodular law recursively we obtain that there are $M_1 , \ldots , M_k \in \mathcal{L}$
with $M_j \leq M_{i}^\bot$ for $i\neq j$  such that for each 
$K \in \{\jmath(A_i ), B_j\colon i,j\}$ there are $\ell_1 , \ell_2 \ldots , \ell_l \in \{1,\ldots,k\}$ such that
$K = \bigvee_{i=1}^{l} M_{\ell_i}$.

Let $Y \subset X$ be the $k$-dimensional subspace of $X$ spanned by the elements $1\otimes M_i$.
We may write \eqref{eq: pm} in the form
\[\pm \sum_\ell c_\ell \otimes M_\ell \sqsubseteq \sum_\ell d_\ell \otimes M_\ell .\]
Let $P_{M_\ell} \colon Y \to [1\otimes M_\ell ]$ be the projection defined above. Since 
$P_{M_\ell}$ is linear and order-preserving, it follows that $|c_\ell | \leq d_\ell$. 
Because of the choice of $M_\ell$:s and the assumption on $\varphi$ we observe that 
\[\left\|\sum_\ell d_\ell \otimes M_\ell \right\| = \left( \sum_\ell d_{\ell}^p \varphi(M_\ell )\right)^{\frac{1}{p}}.\]
Clearly, putting $d_\ell = |c_\ell |$ is the optimal choice of the coefficients for this choice of 
$(M_\ell )$. We obtain 
\[\left\| \sum_i a_i \otimes \j (A_i )\right\|^p \leq \sum_\ell |c_\ell |^p \varphi(M_\ell ) 
=\sum_i |a_i |^p \varphi(\jmath A_i ) \\ =\sum_i |a_i |^p \mu(A_i ) \]
where 
\[\sum_\ell |c_\ell |^p \varphi(M_\ell ) \leq \sum_j |b_j|^p  \varphi( B_j ) .\]
Since $\varepsilon>0$ was arbitrary it follows that 
\[\left\|\sum_i a_i \otimes \j (A_i )\right\| = \left( \sum_i |a_i |^p \mu (A_i)\right)^{\frac{1}{p}}.\]
Thus the investigated linear map is norm-preserving.
\end{proof}

As $(\mathcal{L} ,\varphi)$ may contain such measure algebras in different dispositions, we conclude that 
the corresponding $L^p (\mathcal{L} ,\varphi)$ space is rich in a sense.

\subsection{Disposition of lattices and submeasures}
We note that 
\[\varphi^* (A) := \|1\otimes A\|_{L^1 (\mathcal{L}, \varphi)},\quad A \in \mathcal{L},\]
defines an order-preserving mapping $\varphi^* \colon  \mathcal{L} \to [0,1]$ with $\varphi^* \leq \varphi$,
$\varphi^* ({\mathbf 0})=0$, $\varphi^* ({\mathbf 1})=1$ and 
\[\varphi^* (A \vee B) \leq \varphi^* (A ) +  \varphi^* (B)\]
even if $\varphi$ fails to satisfy the corresponding property, or fails to be order-preserving. Also, if $\mathcal{L}$ is an orthomodular lattice then 
\[\varphi^* (M \vee N) = \varphi^* (M) +  \varphi^* (N)\]
if $N\leq M^\bot$. 
Moreover, it holds that $\|x\|_{L^p (\mathcal{L},\varphi^* )}=\|x\|_{L^p (\mathcal{L},\varphi)}$ for \emph{any} 
mapping $\varphi \colon \mathcal{L} \to [0,1]$ and any $1\leq p < \infty$.
Therefore, is seems reasonable to assume that $\varphi$ satisfies the above properties 
of $\varphi^*$ in the first place.

The finite examples of lattices corresponded to $\ell^p (n)$ spaces and it seems natural to ask 
whether Banach spaces of the type $L^p (\mathcal{L} ,\varphi)$ can be represented in terms of 
classical $L^p (\mu)$ spaces (e.g. as subspaces or quotients). This appears an interesting problem for future
research. 

This calls for the following definition. Suppose that $(\mathfrak{A},\mu)$ is a probability measure algebra and 
$(\mathcal{L},\varphi )$ is a lattice endowed with a submeasure. Then we say that 
$\mathfrak{A}$ is an \emph{algebrification} of $\mathcal{L}$ if one can form the following commuting 
diagram:

\begin{center}
$\begin{CD}
\mathcal{L} 	@>h>>	\mathfrak{A} \\
@VV{\otimes}V		@VV{\otimes\phantom{\ \ \ }}V\\
L^p (\mathcal{L},\varphi )	@>T>> 	L^p (\mathfrak{A},\mu)  .	
\end{CD}$
\end{center}
where $h$ is a lattice homomorphism ($\mathfrak{A}$ being understood as a lattice in the obvious way), 
$\otimes$ denotes the mapping $A \mapsto 1\otimes A$ appearing in the construction of $L^p (\mathcal{L},\varphi )$ spaces and $T$ is an isometric isomorphism between Banach spaces. The definition is not purely algebraic since
it is affected by the (sub)measures of sets (e.g. $h$ necessarily maps $\varphi$-null  elements to the
minimal element of $\mathfrak{A}$). This definition also relies on the facts that the simple
functions are dense in the spaces and that $T$ is an isometry, so that the diagram becomes rather rigid.

\begin{proposition}
Let $(\mathcal{L},\varphi )$ be a lattice endowed with a submeasure and $1\leq p <\infty$, $p\neq 2$. If a corresponding algebrification $\mathfrak{A}$ exists then it is unique.
\end{proposition}
\begin{proof}
The $L^p$-structure of a Banach space is an isometric invariant. In the case with $L^p (\mathfrak{A},\mu)$,
$1\leq p <\infty$, $p\neq 2$, the $L^p$-structure of the space can be identified with $\mathfrak{A}$, see \cite{LPstructure}. Thus $\mathfrak{A}$ is fully determined by $L^p (\mathcal{L},\varphi )$.
\end{proof}

Thus, the most straight-forward representation problem involving $L^p (\mathcal{L},\varphi )$ type spaces
can be formulated as follows: Which submeasured lattices $(\mathcal{L},\varphi )$ admit 
an algebrification?

\subsection*{Acknowledgments}
This research was financially supported by the Academy of Finland Project \#268009, the Finnish Cultural Foundation and V\"{a}is\"{a}l\"{a} Foundation.

\end{document}